\documentclass[12pt]{amsart}
\usepackage{amsthm,amsmath,a4,amssymb,verbatim,url,enumitem,amscd}

\newcommand{\ignorecontentsline}[3]{}

\newtheorem{thm}{Theorem}[section]
\newtheorem{cor}[thm]{Corollary}
\newtheorem{lem}[thm]{Lemma}

\newtheorem{prop}[thm]{Proposition}
\newtheorem{conj}[thm]{Conjecture}
\newtheorem{fact}[thm]{Fact}
\theoremstyle{definition}
\newtheorem{defn}[thm]{Definition}
\theoremstyle{remark}
\newtheorem{rem}[thm]{Remark}

\newtheorem{ex}[thm]{Example}
\newtheorem{que}[thm]{Question}

\newcommand{\Aut}{\textnormal{Aut}}

\newcommand{\slsl}{\rightthreetimes}
\newcommand{\epi}{\twoheadrightarrow}

\newcommand{\N}{\mathbf{N}}
\newcommand{\K}{\mathbf{K}}
\newcommand{\Z}{\mathbf{Z}}
\newcommand{\R}{\mathbf{R}}
\newcommand{\Q}{\mathbf{Q}}

\newcommand{\FC}{\mathsf{B}}
\newcommand{\WW}{\mathsf{W}}

\begin{document}

\address{CNRS and Univ Lyon, Univ Claude Bernard Lyon 1, Institut Camille Jordan, 43 blvd. du 11 novembre 1918, F-69622 Villeurbanne}

\email{cornulier@math.univ-lyon1.fr}

\title{Locally compact wreath products}
\author{Yves Cornulier}%
\date{March 24, 2018}
\thanks{Supported by ANR 12-BS01-0003-01 GDSous}

\begin{abstract}
Wreath products of non-discrete locally compact groups are usually not locally compact groups, nor even topological groups. We introduce a natural extension of the wreath product construction to the setting of locally compact groups.

As an application, we disprove a conjecture of Trofimov, constructing compactly generated locally compact groups of intermediate growth without nontrivial compact normal subgroups.
\end{abstract}


\maketitle

\section{Introduction}

Let $B,H$ be groups and let $X$ be a $H$-set. The {\bf unrestricted wreath product} $B\bar{\wr}_XH$ is the semidirect product $B^X\rtimes H$, where $H$ permutes the copies in the power $B^X$. The (restricted) {\bf wreath product} is its subgroup $B\wr_XH=B^{(X)}\rtimes H$, where $B^{(X)}$ is the restricted power. When $X=H$ with action by left translation, these are called the unrestricted and restricted {\bf standard} wreath product. In both cases, some authors also refer to the standard wreath product simply as ``wreath product", and to the general case as ``permutational wreath product".

Originally the definition comes from finite groups, where the restricted/ unrestricted distinction does not appear. Specifically, the first occurring example was probably the wreath product $C_2\wr_{\{1,\dots,n\}}S_n$, where $C_2$ is cyclic of order 2 and $S_n$ is the symmetric group. It is a Coxeter group of type $B_n/C_n$ and thus a isomorphic to the Weyl group in simple algebraic groups of these types.
 
An early use of general wreath products is the classical theorem \cite{KK} that every group that is extension of a normal subgroup $B$ with quotient $H$ embeds into the unrestricted standard wreath product $B\bar{\wr}H$. See \cite[Th.~6.2]{CC} for a topological version (with $B$ compact and $H$ discrete), as well as \cite{Rem} for a more  subtle generalization in the case of profinite groups.
 
In geometric group theory, the restricted wreath product occurs more naturally: indeed it is finitely generated as soon as $B,H$ are finitely generated and $X$ has finitely many orbits (while the unrestricted is uncountable as soon as $X$ is infinite and $B$ nontrivial).
 
Still, the definition does not immediately generalize to locally compact groups. Indeed, for $X$ infinite, the power $B^X$ fails to be locally compact as soon as $B$ is noncompact, and the restricted power $B^{(X)}$ fails to be locally compact as soon as $B$ is nondiscrete. This bad behavior is essentially well-known. For instance for the standard unrestricted wreath product, it was observed in \cite{D} that, for $B\neq 1$ and $H$ non-discrete the obvious product topology is a not even a group topology on $B\wr H$.

The purpose of this note is to indicate how the definition of wreath products naturally extends, in the context of geometric group theory, to the setting of locally compact groups. This is performed in \S\ref{lcw}.

This extension is natural even within the study of discrete groups. Let us provide three illustrations.

\begin{itemize}[leftmargin=*]
\item It is well-known that for any two finite groups $F_1,F_2$ of the same cardinal $n$, the groups $F_1\wr\Z$ and $F_2\wr\Z$ admit isomorphic (unlabeled) Cayley graphs, just taking $F_i\cup\{1_\Z\}$ as generating subset. This means that these groups admit embeddings as cocompact lattices in a single locally compact group, namely the isometry group of this common Cayley graph. A natural explicit group in which they indeed embed as cocompact lattices is the topological wreath product $\mathfrak{S}_n\wr^{\mathfrak{S}_{n-1}}\Z$ (to be defined in \S\ref{lcw}), where $\mathfrak{S}_k$ is the symmetric group on $k$ letters (see Example \ref{finicay}).

\item Adrien Le Boudec \cite{LB2} uses lattices in such wreath products to obtain two quasi-isometric non-amenable finitely generated groups, one being simple and the other having infinite amenable radical.

\item It is a difficult question to determine which wreath products of discrete groups $G=B\wr_{H/L} H$ have the Haagerup Property, assuming that $B$ and $H$ have the Haagerup Property. It was proved in \cite{CSV2} that this holds if $L=1$. Furthermore, assuming that $L=N$ is normal, we can embed it diagonally into $H\times B\wr_{H/N} (H/N)$; thus if in addition $H/N$ has the Haagerup Property, then $G$ also has the Haagerup Property. Considering topological wreath products allows us to extend this result to the case when $L$ is a commensurated subgroup: 
\begin{thm}\label{bhpro}
Assume that $B,H$ have the Haagerup Property, that $L$ is a commensurated subgroup of $H$ such that the relative profinite completion $H\slsl L$ (which is non-discrete in general, see \S\ref{rpc}) has the Haagerup Property. Then the wreath product $B\wr_{H/L}H$ also has the Haagerup Property.
\end{thm}
(All relevant definitions are given in \S\ref{HaPW}.)
This is a particular case of Theorem \ref{thaag}, which applies to more general (non-discrete) groups. An instance where it applies is when $H=\mathrm{SL}_2(\Z[1/k])$ and $L=\mathrm{SL}_2(\Z[1/\ell])$, where $\ell$ divides $k$. It is not covered by the previously known results (except in the trivial case when $k$ divides some power of $\ell$, in which case $\Z[1/\ell]=\Z[1/k]$). 
Let us also mention that Theorem \ref{thaag} also includes a statement about Property PW, a combinatorial stronger analogue of the Haagerup Property.
\end{itemize}

Finally, using one instance of this wreath product construction, we obtain:

\begin{thm}[See Theorem \ref{bethm}]\label{eig}
There exists a totally disconnected, compactly generated, locally compact group that has subexponential growth and is not compact-by-discrete (i.e., has no compact open normal subgroup).
\end{thm}

This relies on the construction by Bartholdi and Erschler \cite{BE} of some (discrete) wreath products of subexponential growth. This disproves a conjecture of Trofimov \cite[(**), p.~120]{Tro1} about the structure of vertex-transitive graphs, see \S\ref{trocon}. This conjecture is also the main subject of discussion in the more recent \cite{Tro2}.

\medskip

\noindent {\bf Outline.}
\begin{itemize}
\item In \S\ref{lcw} we introduce semirestricted wreath products, which is the promised natural extension of wreath products to the setting of locally compact groups.
\item In \S\ref{HaPW} we prove a stability result for the Haagerup Property and its combinatorial strengthening Property PW, including Theorem \ref{bhpro} as a particular case.
\item In \S\ref{poly}, we describe the subgroup of bounded element and the polycompact radical for arbitrary semirestricted locally compact wreath products. This is used in a very particular case to obtain that this subgroup is trivial, so as to prove that one group is not compact-by-discrete, in \S\ref{intgro}.
\item In \S\ref{intgro}, we prove Theorem \ref{eig}, explain why it disproves Trofimov's conjecture, and ask some further open questions on locally compact groups of intermediate growth.
\item In \S\ref{presen}, we extend some results of infinite presentability to the locally compact setting, and we also consider them in the analogous context of wreathed Coxeter groups.
\item In \S\ref{otherw}, we present a variant of the construction of \S\ref{lcw}, relying on a commensurating action of the acting group.
\end{itemize}

\bigskip

\textbf{Acknowledgement.} I thank Adrien Le Boudec for his interest and motivating discussions. I thank Pierre-Emmanuel Caprace for pointing me out  Trofimov's conjecture as well as useful remarks. I thank the referee for a careful reading and useful references.

\section{Wreath products in the locally compact context}\label{lcw}

We wish to extend the wreath product $B\wr_X H$ to locally compact groups. In all the following, $B,H$ are groups and $X$ is an $H$-set; the group $H$ acts on $B^X$ by $h\cdot f(x)=f(h^{-1}x)$.

We begin with the easier case when $B$ is still assumed to be discrete; now $H$ is a locally compact group. On the other hand, we still assume that $X$ is discrete: $X$ is a continuous discrete $H$-set. This means that one of the following equivalent conditions is fulfilled:
\begin{itemize}
\item the action of $H$ on $X$ is continuous (that is, the action map $H\times X\to X$ is continuous)
\item for every $x\in X$, the stabilizer $H_x$ is open in $H$;
\item $X$ is isomorphic as an $H$-set to a disjoint union $\bigsqcup_{i\in I}H/L_i$, where $(L_i)_{i\in I}$ is a family of open subgroups of $H$.
\end{itemize}

Then the action of $H$ on the discrete group $B^{(X)}$ (restriction of the action on $B^X$) is continuous. Then the semidirect product $B^{(X)}\rtimes H$ is a locally compact group for the product topology. Note that when $H$ is non-discrete, this does not include the standard wreath product; in this setting, the closest generalization is the case when $X=H/L$ with $L$ compact open. 

Now let us deal with the general case. We know that the restricted wreath product behaves well when $B$ is discrete and the unrestricted wreath product behaves well when $B$ is compact. The natural definition consists in interpolating between restricted and unrestricted wreath products. We first define it with no topological assumption:

\begin{defn}Let $B,H,X$ be as above and let $A$ be a subgroup of $B$.
We first define the {\bf semirestricted power}
\[B^{X,A}=\{f\in B^X:\;f(x)\in A,\;\forall^* x\in X\},\] 
where $\forall^*$ means ``for all but finitely many". The {\bf semirestricted wreath product} is defined as the semidirect product
\[B\wr^A_XH=B^{X,A}\rtimes H.\]
\end{defn}

\begin{rem}
The semirestricted power is a particular case of the (semi)restricted product, which underlies the classical notion of Adele group, and is also considered more generally (and with closer motivation) in \cite[\S 3.1]{BCGM},  to notably construct non-cocompact lattices in some metabelian locally compact groups. 

An instance of semirestricted power (attributed to the author) appears in a paper of Eisenmann and Monod \cite[\S 3]{EM}: namely, $F$ is a finite perfect group and $K$ a nontrivial subgroup such that $F$ is not normally generated by any element of $K$, but that generates $F$ normally. Then for $X$ infinite, $F^{X,K}$ is a perfect locally compact group, with no infinite discrete quotient, that is not topologically normally generated by any element.

Locally compact wreath products $B\wr_{H/L}H$ with $B$ discrete, but $H$ arbitrary, are mentioned in \cite[Theorem C]{GM}.
 
An instance of semirestricted wreath product is mentioned in \cite[Prop.\ 6.14]{LB}.
\end{rem}

Next, the general definition in the locally compact setting is when $A$ is compact open in $B$.

The following lemma is standard. 
\begin{lem}[\cite{CH}, Prop.\ 8.2.4]\label{extop}
Let $G$ be a group and $H$ a subgroup. Let $\mathcal{T}$ be a group topology on $H$. Suppose that every conjugation in $G$ restricts to a continuous isomorphism between two open subgroups of $H$. Then there is a unique topology on $G$ making $H$ open with the induced topology coinciding with $\mathcal{T}$. \qed
\end{lem}

\begin{prop}
Suppose that $B$ is a locally compact group and $A$ a compact open subgroup and $X$ a (discrete) set. There is a unique structure of topological group on $B^{X,A}$ that makes the embedding of $A^X$ a topological isomorphism to an open subgroup. It is locally compact.

Suppose in addition that $H$ is a locally compact group and that the $H$-action on $X$ is continuous (i.e., has open point stabilizers). Then there is a unique structure of topological group on $B\wr^A_XH=B^{X,A}\rtimes H$ that makes it a topological semidirect product. 
\end{prop}
\begin{proof}
The first fact follows from Lemma \ref{extop}.
Let us now check that $H$ acts continuously on $B^{X,A}$. If $h_i\to h$, $w_i\to w$, then for $i$ large enough, we can write $w_i=\alpha_iw$ with $\alpha_i\in A^X$, $\alpha_i\to 1$. Also write $h_i=m_ih$, where $m_i\to 1$. Then 
\[h_i\cdot w_i= h_i\cdot (\alpha_iw)=(h_i\cdot\alpha_i)(h_i\cdot w)\]
\[(m_i\cdot(h\cdot\alpha_i))(m_i\cdot (h\cdot w)).\]
We have $\nu_i=h\cdot\alpha_i\to 1$, and hence $m_i\cdot\nu_i\to 1$: indeed, for all $x\in X$, $m_i\cdot \nu_i(x)=\nu_i(m_i^{-1}x)$; since $m_i\to 1$, for large $i$ we have $m_i^{-1}x=x$, whence the fact. Similarly $m_i\cdot (h\cdot w)$ tends to $h\cdot w$ and since $B^{X,A}$ is a topological group, we obtain $h_i\cdot w_i\to h\cdot w$.
\end{proof}

Recall that a homomorphism between locally compact groups is copci if it is continuous, proper, with cocompact image.

\begin{prop}\label{copcip}
Let $B_1,B_2$ be locally compact groups with compact open subgroups $A_1,A_2$, and $u:B_1\to B_2$ be a continuous homomorphism mapping $A_1$ into $A_2$. This yields continuous homomorphisms 
\[f:B_1^{X,A_1}\to B_2^{X,A_2},\qquad f':B_1\wr^{A_1}_XH\to B_2\wr^{A_2}_XH.\]
Consider the induced map $\bar{u}:B_1/A_1\to B_2/A_2$.
\begin{enumerate}
\item Suppose that $X\neq\emptyset$. Then $f$ is proper if and only if $f'$ is proper, if and only if $\bar{u}$ is injective (that is, $u^{-1}(A_2)=A_1$).

\item Suppose that $X$ is infinite. Then $f$ has cocompact image if and only $f'$ has cocompact image, if and only if $\bar{u}$ is surjective (that is, the composite map $B_1\to B_2\to B_2/A_2$ is surjective).

\item Thus for $X$ infinite, $f$ is copci if and only if $\bar{u}$ is bijective.
\end{enumerate}
\end{prop}
\begin{proof}
It is enough to check everything for $f$. The inverse image of the compact open subgroup ${A_2}^X$ is $(u^{-1}(A_2))^{X,A_1}$, and is compact if and only if $u^{-1}(A_2)=A_1$. This yields the first part. 

For the second, if $M$ is the projection in $B_2/A_2$ of the image, then the projection in $(B_2/A_2)^{(X)}=B_2^{X,A_2}/A_2^X$ of the image is $M^{(X)}$. It has finite index only when $M=B_2/A_2$, and since $A_2^X$ is compact, conversely if $M=B_2/A_2$ then cocompactness follows.
\end{proof}

\begin{rem}
Since $A$ is a compact open subgroup of $B$, the unit connected component $B^\circ=A^\circ$ is compact. We readily see that $(B\wr^A_X H)^\circ$ is the unrestricted wreath product $B^\circ\bar{\wr}_X H^\circ\simeq (B^\circ)^X\times H^\circ$.
\end{rem}

 \begin{ex}\label{finicay}
Fix any finitely generated group $\Gamma$. Let $F$ be a finite group, and fix a bijection of $F$ with $\{1,\dots,n\}$. The left action of $F$ on itself yields a homomorphism $F\to\mathfrak{S}_n$, inducing a bijection $F\to\mathfrak{S}_n/\mathfrak{S}_{n-1}$, where $\mathfrak{S}_{n-1}$ is any point stabilizer. Hence, by Proposition \ref{copcip}, it induces an embedding of $F\wr\Gamma$ into $\mathfrak{S}_{n}\wr^{\mathfrak{S}_{n-1}}\Gamma$ as a cocompact lattice.

It would thus be interesting to further investigate these groups $\mathfrak{S}_{n}\wr^{\mathfrak{S}_{n-1}}\Gamma$.

Actually, the isometry group $H$ of the Cayley graph of $\Gamma$, with respect to some finite generating subset $S$, can turn out to be larger than $\Gamma$, e.g., non-discrete (e.g., when $\Gamma$ is free over $S$ and $S$ contains at least two elements). Then, the isometry group of the Cayley graph of $F\wr\Gamma$ (with respect to $F\cup S$) includes a larger subgroup including $\mathfrak{S}_{n}\wr^{\mathfrak{S}_{n-1}}\Gamma$, namely $\mathfrak{S}_{n}\wr^{\mathfrak{S}_{n-1}}_L H$, where $L$ is the (compact) stabilizer of $1\in\Gamma$ in the isometry group of the Cayley graph of $\Gamma$. We can expect this to often coincide with the full isometry group of the given Cayley graph of $F\wr\Gamma$. 

Note that for the wreath product $C\wr\Z$, this only yields an embedding into an overgroup of finite index, while there are natural known non-discrete envelopes, see Example \ref{envw2} for $q=2$.
 \end{ex}

The semirestricted wreath product construction preserves unimodularity:

\begin{prop}\label{unimod}
The semirestricted wreath product $B\wr^A_XH$ is unimodular as soon as $B$ and $H$ are unimodular.
\end{prop} 
\begin{proof}
Since $A$ is open in $B$ and $B$ is unimodular, $A$ is unimodular and conjugation by $B$ preserves locally its Haar measure around 1. It follows that conjugation of $A^X$ by an element of the form $\delta_x(b)$ (mapping $x$ to $b$ and other elements of $X$ to 1) preserves locally the Haar measure of $A^X$, and hence this holds for all elements of $B^{X,A}$. Hence $B^{X,A}$.

Clearly $H$ preserves the Haar measure of $A^X$, and hence it locally preserves the Haar measure of $B^{X,A}$. Since $H$ is unimodular, it follows that the semidirect product is also unimodular.
\end{proof} 

\begin{rem}
Some authors, such as Klopsch \cite[\S 4.3]{Kl} refer to a possible notion of profinite wreath product, defining it in a particular case.

Let us define it here, calling it {\bf compact wreath product}; it will not be used elsewhere in the paper and is distinct from the constructions we consider. Namely, let $B$ be an {\it abelian} compact group and $H$ a profinite group. The compact wreath product $B\hat{\wr}H$ is the projective limit of $B\bar{\wr}P$, where $P$ ranges over (Hausdorff) finite quotients of $P$. It is mentioned in the particular case: $B$ cyclic of order $p$ and $H$ the $p$-adic group $\Z_p$ in \cite[\S 4.3]{Kl}; the same example also occurs in \cite{DS}.

The main drawback of this construction, namely the restriction to $B$ abelian, is due to the fact that when $P$ is a finite quotient of $K$ and $P'$ a quotient of $P$, there is a canonical homomorphism from $B\wr P$ onto  $B\wr P'$ only when $B$ is abelian. Its main advantage is that it does not refer to any choice of open subgroup in $B$ and does not require that $H$ is discrete.
\end{rem}
 
\section{Haagerup and PW Properties}\label{HaPW}
 
\subsection{Definition of Haagerup and PW Properties} 
 
Recall that locally compact group has the {\bf Haagerup Property} if the function 1 on $G$ admits an approximation, uniformly on compact subsets, by continuous positive definite functions vanishing at infinity. A locally compact group has the Haagerup Property if and only if all its open, compactly generated subgroups have the Haagerup Property. For these compactly generated subgroups, or more generally for $\sigma$-compact locally compact groups, the Haagerup Property is equivalent to the existence of a proper continuous conditionally negative definite function, or equivalently of a metrically proper affine isometric action on a Hilbert space. All these facts are due to Akemann and Walter \cite{AW}. 

Also recall that a locally compact group has {\bf Property PW} if it admits a continuous action on a discrete set $X$ with a subset $M\subset X$ such that the function $g\mapsto\ell(g)=\#(M\triangle gM)$ takes finite values and is proper. This notion has been widely considered (at least for discrete group actions) before being given a name in \cite{CSV1} and being studied in \cite{CFW} (where in particular it is checked that $\ell$ is automatically continuous). Clearly Property PW implies $\sigma$-compactness. It is equivalent to the existence of a metrically proper continuous action on a CAT(0) cube complex.

\subsection{Commensurated subgroups and relative profinite completion}\label{rpc}

Recall that two subgroups of a group are {\bf commensurate} if their intersection has finite index in both; a subgroup is {\bf commensurated} if its conjugates are pairwise commensurate. It is a classical observation that a subgroup $L$ of a group $H$ is commensurated if and only if $L$ has finite orbits on $H/L$.

\begin{defn} Let $H$ be a locally compact group and $L$ be an open subgroup of $H$. The {\bf relative profinite completion} of $(H,L)$ is the projective limit $H\slsl L=\underleftarrow{\lim}\,H/M$, where $M$ ranges over the open finite index subgroups of $L$.
\end{defn}
Endowed with the projective limit topology, this is a locally compact space (regardless of $L$ being commensurated) with continuous (left) $H$-action and a canonical continuous $H$-equivariant map $H\to H\slsl L$. If, in addition, $L$ is commensurated, then there is a unique continuous group law making it a group homomorphism.

The closure $L\slsl L$ of the image of $L$ in $H\slsl L$ is an open subgroup and $L\to L\slsl L$ is the usual profinite completion of $L$. Note that the induced map $H/L\to (H\slsl L)/(L\slsl L)$ is a bijection, so we denote the latter as $H/L$ if necessary.

\begin{rem}
a) The relative profinite completion was introduced by Belyaev \cite{Be}; it is sometimes referred as Belyaev completion, or profinite completion of $H$ localized at the subgroup $L$ (all references I am aware of assume that $H$ is discrete).

b) Let now again $L$ be an open commensurated subgroup of $H$. A closely related construction is the {\bf Schlichting completion} $H\slsl_{\mathrm{Schl}}L$: this is the closure of the image of $H$ in the Polish group of permutations of $H/L$. Actually, this is canonically the quotient of the relative profinite completion by the core of $L\slsl L$ in $H\slsl L$ (that is, the largest normal subgroup of $H\slsl L$ included in the closure $L\slsl L$ of the image of $L$). In particular, this canonical quotient map $H\slsl L\to H\slsl_{\mathrm{Schl}}L$ has a compact kernel. (See for instance \cite[\S 4]{RW} for these facts.)

c) The Schlichting completion is sometimes called relative profinite completion. This choice is, in my opinion, confusing and we did not follow it, notably because $H\slsl H$ is the profinite completion, while $H\slsl_{\mathrm{Schl}}H$ is the trivial group.
An object called ``relative profinite completion" should include the usual profinite completion as a particular case. The Schlichting completion is a practical construction, but is a derived object of the more fundamental relative profinite (or Belyaev) completion.
\end{rem}

\subsection{The stability results}
 
\begin{thm}\label{thaag}
Consider a semirestricted wreath product $G=B\wr^A_X H$ ($H,B$ are locally compact group, $X$ an $H$-set with open point stabilizers, $A$ is a compact open subgroup of $B$).

1) Assume that $B,H$ have the Haagerup Property. Assume that for every point stabilizer $L$ of $X$, $L$ is a commensurated subgroup of $H$ and the relative profinite completion $H\slsl L$ has the Haagerup Property. Then $G$ has the Haagerup Property.

2) Assume that $G$ is compactly generated (when $X$ is non-empty and $A\neq B$, this means that $H,B$ are compactly generated and $X$ has finitely many $H$-orbits). Then the previous statement holds true when all occurrences of ``Haagerup Property" are replaced with ``Property PW".
\end{thm}

The proof will be by reduction to the case when $X=H/L$, with $L$ compact open. This latter case can be viewed as a locally compact extension of the case of standard wreath products (slightly more general in the discrete case, where it covers the case with finite stabilizers).

\begin{lem}\label{haag_compact}
Consider a semirestricted wreath product $G=B\wr^A_{H/L} H$, with $L$ compact open in $H$. Assume that $B,H$ have the Haagerup Property. Then $G$ has the Haagerup Property. 
\end{lem}
\begin{proof}

First recall that a $H$-invariant walling on $H/L$ means an $H$-invariant Radon measure $\mu$ on the locally compact space $2^{H/L}_*=2^{H/L}\smallsetminus\{\emptyset,H/L\}$. To such a walling, we can associate the pseudo-distance $d_\mu$ on $H/L$ defined by
\[d_\mu(x,x')=\mu\big\{M\subset H/L:\; M\vdash \{x,x'\}\big\}.\]
Here $\vdash$ reads as ``cuts" and $M\vdash N$ means that both intersections $M\cap N$ and $M\cap N^c$ ($N^c$ being the complement) are non-empty.

Let us now prove the result. Using that the Haagerup property is stable under directed unions of open subgroups \cite[Prop.\ 6.1.1]{CCJJV}, we first reduce to the case when $G$ is $\sigma$-compact. Namely, we need to show that every $\sigma$-compact open subgroup $U$ of the given semirestricted product $G=B^{H/L,A}\rtimes H$ is included in an open subgroup that is also, after modding out by a compact normal subgroup, a semirestricted product of the same form, but in addition $\sigma$-compact. (Taking the quotient by a compact normal subgroup does not matter because the Haagerup property is clearly invariant under taking extensions by compact kernels.)

Denote by $(B/A)^{(H/L)}$ the set of functions from $H/L$ to $B/A$ with finite support, where the support is the set of points on which the value differs from the basepoint of $B/A$ (that is, the trivial left $A$-coset). The projection $B\to B/A$ induces a continuous projection $\rho:B^{H/L,A}\to (B/A)^{(H/L)}$. Also denote by $\pi_1$ and $\pi_2$ the projections from $B^{H/L,A}\rtimes H$ to $B^{H/L,A}$ and $H$ defined by $\pi_1(f,h)=f$, $\pi_2(f,h)=h$. These are continuous maps.

Since $U$ is $\sigma$-compact and $(B/A)^{(H/L)}$ is discrete, $\rho(\pi_1(U))$ is countable,  and hence included in $(B_1/A)^{(V/L)}$ for some open, $\sigma$-compact subgroup $B_1$ of $B$, including $A$, and some open, $\sigma$-compact subgroup $V$ of $H$, including $L\cup\pi_2(U)$. Then $U$ is included in the open subgroup $B_1\wr_{H/L}^AV$, and actually sits in a smaller open subgroup, namely the semidirect product $G_1=(B_1^{(V/L)}A^{H/L})\rtimes V$. Define $Y=H/L\smallsetminus V/L$. Then 
\[G_1=(B_1^{V/L,A}\times A^Y)\rtimes V.\] Observing that $Y$ is $V$-invariant, we deduce that $A^Y$ is a compact normal subgroup of $G_1$ and $G_1/A^Y$ is isomorphic to the $\sigma$-compact semirestricted wreath product $B_1\wr_{V/L}^AV$. This terminates the proof of the reduction to the case when $H$ and $B$ are $\sigma$-compact.

Since $H$ has the Haagerup Property and is $\sigma$-compact, it admits a continuous proper conditionally definite function. By averaging by the compact subgroup $L$, we can choose it to be right $L$-invariant. Thus it yields a proper $H$-invariant kernel $\psi$ on $H/L$. Set $\kappa=\sqrt{\psi}$. By \cite[Proposition 2.8(iii)]{CSV2} (which is strongly inspired by Robertson-Steger \cite[Proposition 1.4]{RS}), there exists an $H$-invariant walling $\mu$ on $H/L$ such that $\kappa=d_\mu$. 

For $f\in B^{H/L,A}$, define $\mathrm{Supp}_A(f)=\{x\in H/L:f(x)\notin A\}$. Define $D_\mu:G\times G\to\R$ as follows: for $f_i\in B^{H/L,A}$, $h_i\in H$, $i=1,2$,
\[D_\mu(f_1h_1,f_2h_2)=\mu\big\{M\subset H/L: M\vdash\mathrm{Supp}_A(f_1^{-1}f_2)\cup\{h_1L,h_2L\}\big\}.\]

The difference with the case of standard wreath products from \cite{CSV2} is that we have $\{h_1L,h_2L\}$ instead of $\{h_1,h_2\}$ and $\mathrm{Supp}_A$ instead of $\mathrm{Supp}$. The proof follows, however, the same lines.

Then $D_\mu$ is well-defined, continuous, and left-invariant. It is well-defined because the set $\mathrm{Supp}_A(f_1^{-1}f_2)\cup\{h_1L,h_2L\}$ is finite and the set of $M$ cutting a given finite subset is a compact open subset of $2^{H/L}_*$. It is continuous because it is locally constant. It is obviously left-invariant by $B^{H/L,A}$. Finally, it is $H$-invariant, by an immediate verification using that $\mu$ is $H$-invariant.

Recall that a map $\Psi:Y\times Y\to\R$ is measure-definite if it is $L^1$-embeddable, in the sense that there is a map $\varphi$ from $Y$ to some $L^1$-space such that $\Psi(y,y')=\|\varphi(y)-\varphi(y')\|$ for all $y,y'\in Y$ (that is, $\Psi$ is isometric, although not necessarily injective). It is well-known that measure-definite implies conditionally negative definite (see \cite{RS}, from which the terminology is borrowed).

We claim that $D_\mu$ is measure-definite, that is, isometrically embeddable into $L^1$ (this, regardless of the assumption that $\mu$ is $H$-invariant). This fact being closed under combinations and pointwise limits, it is enough to check it when $\mu=\delta$ is a Dirac measure at some $M\subset H/L$. Here $D_\delta(f_1h_1,f_2h_2)$ is $1$ or 0 according to whether $M$ cuts $\mathrm{Supp}_A(f_1^{-1}f_2)\cup\{h_1L,h_2L\}$. Being a $\{0,1\}$-valued kernel, that $D_\delta$ is a pseudo-distance is equivalent to the condition that being at $D_\delta$-distance zero is an equivalence relation; the easy argument, already performed in \cite{CSV2}, is left to the reader. Next, a $\{0,1\}$-valued pseudo-distance is obviously measure-definite.

If $D_\mu(1,fh)\le n$, then $\mu\{M:M\vdash\mathrm{Supp}_A(f)\cup\{L,hL\}\}\le n$. 
If $F$ is a finite subset, then $\mu\{M:M\vdash F\}\ge\sup_{x,y\in F}\mu\{M:M\vdash\{x,y\}$. Hence for all $x,y\in \mathrm{Supp}_A(f)\cup\{L,hL\}$, we have 
$d_\mu(x,y)\le n$. In particular, $\mathrm{Supp}_A(f)\cup\{hL\}$ is included in the $n$-ball $B_n$ of $H/L$, which is finite by properness.

To conclude, we combine with another affine action. Using that $B$ is $\sigma$-compact, we can fix one continuous proper affine isometric action of $B$ on a Hilbert space $\mathcal{H}$; we can suppose that $A$ fixes 0. Let $\xi$ be the corresponding conditionnally negative definite function on $B$. Then let $B^{H/L,A}$ act on the $\ell^2$-sum $\bigoplus_{x\in H/L}\mathcal{H}_x$ of copies of $x$, the $x$-th component in $B^{H/L,A}$ acting on the $x$-component of the direct sum by the given action, and trivially on other components. The action extends to an action of the semirestricted wreath product, $H$ permuting the components. The resulting conditionally negative definite function is $\xi'(fh)=\sum_{x\in H/L}\xi(f(x))$; it is continuous. 

If $\xi'(fh)\le n$, then $\xi(f(x))\le n$ for all $x$. In particular, if both $D_\mu(1,fh)\le n$ and $\xi'(fh)\le n$, then $\mathrm{Supp}_A(f)\subset B_n$ and in $B_n$, $f$ takes values in the compact subset $\xi^{-1}([0,n])$ and $h\in B_n$. We see that this forces $fh$ to belong in some compact subset of $G$. Thus the continuous conditionally negative definite function $fh\mapsto D_\mu(1,fh)+\xi'(fh)$ is proper on $G$.
\end{proof}

\begin{lem}\label{pw_compact}
Consider a semirestricted wreath product $G=B\wr^A_{H/L} H$, with $L$ compact open in $H$. Assume that $B,H$ have the Property PW. Then $G$ has Property PW. 
\end{lem}

\begin{proof}
The proof could be expected to work along the same lines as for Property PW. The problem is that for Property PW, we do not know if it is enough to have a left-invariant continuous proper distance that is a sum of cut-metrics: Property PW indeed requires that the decomposition under cut-metrics be invariant by left-translation, which is not a priori clear. This means that we have to be more explicit, essentially following the discrete case \cite{CSV1}, or rather its version in terms of commensurating actions given in \cite[Proposition 4.G.2]{CFW}.

Consider a continuous discrete $H$-set $Y$ with commensurated subset $M$, in the sense that $\ell_0(h)=M\bigtriangleup hM$ is finite for all $h\in H$. We will eventually assume that $\ell_0$ is proper on $H$. We can suppose that $M$ is $L$-invariant. For $y\in Y$, define $W_y=\{hL\in H/L:y\in hM\}$ (indeed $y\in hM$ is a right-$H$-invariant condition on $h$). Let $Z$ be the set of pairs $(y,p)$, where $y\in Y$ and $p$ is a function $H/L\to B/A$ with $\mathrm{Supp}(p)$ finite and included in the complement $W_x^c$ of $W_x$.

We want to define an action of $B\wr^A_{H/L}H$ on $Z$. We begin defining 
\[h\cdot (y,p)=(hy,h\cdot p);\quad f\cdot (y,p)=(y,\bar{f}|_{W_y^c}p),\quad h\in H,f\in B^{H/L,A},\]
$\bar{f}$ being the image of $f$ in $(B/A)^{(H/L)}$.
First, for $(y,p)\in Z$, observe that \[\mathrm{Supp}_A(h\cdot p)=h\,\mathrm{Supp}_A(p)\subset hW_y^c=W_{hy}^c.\] 
Thus $h\cdot (y,p)\in Z$. That $f\cdot(y,p)\in Z$ is clear. That these maps define actions of $H$ and $B^{H/L,A}$ is clear; both actions have open stabilizers and thus are continuous.
To see that this extends to an action of the topological semidirect product, we have to compute, for $h\in H,f\in B^{H/L,A}$ 
\[h\cdot f\cdot h^{-1}\cdot(y,p)=h\cdot f\cdot(h^{-1}y,h^{-1}\cdot p)\]
\[=h\cdot (h^{-1}y,\bar{f}|_{W^c_{h^{-1}y}}h^{-1}\cdot p)=(y,h\cdot\bar{f}|_{W^c_{h^{-1}y}}p);\]
we have $h\cdot\bar{f}|_{W^c_{h^{-1}y}}=(h\cdot f)|_{hW^c_{h^{-1}y}}=(h\cdot f)|_{W^c_y}=(hfh^{-1})|_{W^c_y}$. Thus $h\cdot f\cdot h^{-1}\cdot(y,p)=(hfh^{-1})\cdot (y,p)$ and this duly yields a continuous action of the semidirect product.

Define $N=M\times\{1\}$, $\ell_0(h)=\#(M\bigtriangleup hM)$ and $\ell(g)=\#(N\bigtriangleup gN)$. Let us check that $\ell$ takes finite values.
By subadditivity, it is enough to check that $\ell$ takes finite values on both $H$ and $B^{H/L,A}$. For $h\in H$, we have $hN\bigtriangleup N=(hM\bigtriangleup M)\times\{1\}$ and hence $\ell$ is finite on $H$, where it coincides with $\ell_0$.
We have $f\cdot (y,p)\in N$ if and only if $y\in M$ and $\bar{f}_{W^c_y}p=1$.
Thus $(y,p)\in N\smallsetminus f^{-1}N$ if and only if $y\in M$, $p=1$, and $\bar{f}_{W^c_y}\neq 1$. The latter means that $W^c_y\cap\mathrm{Supp}_A(f)\neq\emptyset$. For a given element $hL\in \mathrm{Supp}_A(f)\subset H/L$, the condition $hL\in W^c_y$ means $hL\notin W_y$, that is, $y\notin hM$. Thus we have
\[N\smallsetminus f^{-1}N=\{(m,1):m\in\bigcup_{hL\in \mathrm{Supp}_A(f)}M\smallsetminus hM.\}\]
Thus $f^{-1}N\smallsetminus N=f^{-1}(N\smallsetminus fN)$ is also finite. So $\ell$ takes finite values on both $B^{H/L,A}$ and $H$, and hence on all of $G$, by sub-additivity.

We have $(M\bigtriangleup hM)\times\{1\}\subset N\bigtriangleup fhN$. Thus $\ell(fh)\ge\ell_0(h)$.
Also since for each fixed $hL\in\mathrm{Supp}_A(f)$, the subset $N\smallsetminus f^{-1}N$ includes $(M\smallsetminus hM)\times\{1\}$, we have $\ell(f)\ge \#(M\smallsetminus hM)$.

We can assume from the beginning that $M\smallsetminus hM$ and $hM\smallsetminus M$ have the same cardinal for all $h$ (replace if necessary $Y$ with $Y\times\{0,1\}$ and $M$ with $M\times\{0\}\cup M^c\times\{1\}$).
Under this assumption (made for convenience), we have $2\ell(f)\ge \ell(h)$ for all $hL\in\mathrm{Supp}_A(f)$. In other words, $\mathrm{Supp}_A(f)$ is included in the set of $hL$ such that $\ell_0(h)\le 2\ell(f)$. Assuming from now one that $\ell_0$ is proper, this set is finite.

Let $(f_nh_n)$ be a sequence in $G$, with $f_n\in B^{H/L,A}$ and $h_n\in H$, such that $\ell(f_nh_n)$ is bounded, the property $\ell(f_nh_n)\ge\ell_0(h_n)$ and the properness of $\ell_0$ ensures that $(h_n)$ is bounded. Hence $\ell(f_n)\le\ell(f_nh_n)+\ell(h_n)$ is also bounded, say by $k_0$. Let $F(k_0)$ be the set of $hL\in H/L$ such that $\ell_0(h)\le 2k_0$; it is finite by properness of $\ell_0$. Then $\mathrm{Supp}_A(f_n)$ is included in $F(k_0)$. 

If $B$ is compact, this shows that $\ell$ is proper on $G$ and we are done. In general, we consider another, simpler commensurating action of $G$, so as to also make use of the assumption that $H$ has Property PW. Namely, start from a commensurating action of $B$, say on $Y'$, commensurating a subset $M'$, such that the function $\ell_1:b\mapsto \#(M'\bigtriangleup bM')$ is proper on $B$. We can suppose that $M'$ is $A$-invariant.

Let $B^{H/L,A}$ act on $Z'=Y'\times H/L$ by $f\cdot (y,x)=(f(x)y,x)$ (we can think of $Z'$ as disjoint copies of $Y'$ on which the factors in $B^{H/L,A}$ act separately). This action has open stabilizers and hence is continuous. It commensurates the subset $N'=M'\times H/L$. The group $H$ also acts on $Z'$ by permuting copies; this action is continuous and preserves $M'$; clearly it extends the semirestricted wreath product. For $g\in B\wr^A_{H/L}H$, define $\ell'(g)=\#(M'\bigtriangleup gM')$. Then $\ell'(fh)=\ell'(f)$ for all $f\in B^{H/L,A}$ and $h\in H$. Precisely, $\ell'(fh)=\sum_{\gamma L\in H/L}\ell_1(f(\gamma L))$. In particular, if $\ell'(fh)\le k_1$, then $f$ is valued in $K_1=\ell_1^{-1}([0,k_1])$, which is compact by properness of $\ell_1$. 

Therefore, if $\ell(f_nh_n)+\ell'(f_nh_n)$ is bounded, then $(h_n)$ is bounded, and for some $k_0$ as above, $\mathrm{Supp}_A(f_n)$ is included in $F(k_0)$ and $f$ takes values in some compact subset $K_1$. This means that $f_nh_n$ stays in some compact subset of $B\wr^A_{H/L}H$, showing the properness of $\ell+\ell'$.
\end{proof}

\begin{proof}[Proof of Theorem \ref{thaag}]
If $X=\emptyset$ or $B=A$, the quotient of $G$ by the compact normal subgroup $B^X$ is isomorphic to $H$, and the Haagerup/PW Property follows. We can thus suppose, in each case, that $X\neq\emptyset$ and $B\neq A$.

Recall that a group has the Haagerup Property if and only if all its compactly generated open subgroups do. So we can suppose that $G$ is compactly generated (in the PW case this is an assumption); in particular $H,B$ are compactly generated and $X$ has finitely many $H$-orbits: $X=\bigsqcup_{i=1}^n X_i$. Then $G$ embeds as a closed subgroup in $\prod_{i=1}^n B\wr^A_{X_i}H$, which reduces to the transitive case: $X=H/L$. By assumption, $L$ is commensurated.

Consider the diagonal embedding $G\to (B\wr^A_{H/L}(H\slsl L))\times H$.
It is proper: indeed, if $Q$ is a compact neighborhood of 1 in $H$, the inverse image of 
the compact neighborhood of $(A^{H/L}\times (L\slsl L))\times Q$ of 1 is included in the subset $A^{H/L}\times Q$. (We use here that a continuous homomorphism between locally compact group is proper if and only if the inverse image of some compact neighborhood of 1 is compact.)

Therefore, it is enough to show that $B\wr^A_{H/L}(H\slsl L)$ has the Haagerup Property, resp.\ Property PW. In other words, we are reduced to the case when $L$ is compact. This is the contents of Lemma \ref{haag_compact} in the Haagerup case and Lemma \ref{pw_compact} in the PW case. 
\end{proof}

We leave the converse of Theorem \ref{thaag}, for the Haagerup Property, as a conjecture.

\begin{conj}\label{conj_hagrec}
Under the assumptions that all point stabilizers are commensurated subgroups of $H$ (and assuming $A\neq B$), the converse holds: if $B\wr^A_X H$ has the Haagerup Property, then for every point stabilizer $L$ of $X$, the relative profinite completion $H\slsl L$ has the Haagerup Property.
\end{conj}

\begin{rem}\label{reduc}
Conjecture \ref{conj_hagrec} can be reduced to the case when $A=1$ and $B$ is discrete cyclic of prime order, and $X=H/L$ is a transitive $H$-set. 

Indeed, we can first choose $b\in B\smallsetminus A$ and consider the closure of the subgroup it generates, to assume that $B$ has a dense cyclic subgroup. This implies that either $B$ is infinite cyclic, or compact abelian. Then, in the latter case $A^X$ is a compact normal subgroup in the semirestricted wreath product, and hence we can mod it out and still preserve the Haagerup Property. Therefore, we can suppose $A=1$. Thus we can suppose that $B$ is discrete and cyclic (infinite or prime order if necessary), and $G=B\wr_{X}H$. Also, if the stabilizer condition fails for some $x$, we can pass to the closed subgroup $B\wr_{Y}H$, where $Y$ is the $H$-orbit of $x$. 
\end{rem}

\begin{prop}
Conjecture \ref{conj_hagrec} holds when $L$ is a normal subgroup.
\end{prop}
\begin{proof}
By Remark \ref{reduc}, we can assume that $A=1$ and $B$ is discrete abelian nontrivial. By assumption, $G=B^{(H/L)}\rtimes H$ has the Haagerup Property. Thus $(B^{(H/L)}\rtimes H,B^{(H/L)})$ has the relative Haagerup Property (assuming as we can that $G$ is $\sigma$-compact, this means that it admits a continuous affine isometric action on a Hilbert space that is proper in restriction to $B^{(H/L)}$. By \cite[Corollary 5]{CT}, it follows that $(B^{(H/L)}\rtimes (H/L),B^{(H/L)})$ also have the relative Haagerup Property. Let $\psi$ be a continuous conditionally positive definite function on $G$. If $(h_n)$ is any sequence in $H/L$ leaving compact subsets and $b$ is a nontrivial element in $B$, and $\delta$ the function on $H/L$ mapping 1 to $b$ and other elements to 1, then $h_n\delta h_n^{-1}\in B^{(H/L)}$ tends to infinity. Since $\psi$ is proper on $B^{(H/L)}$, we obtain that $\psi(h_n\delta h_n^{-1})$ tends to infinity, and hence (since $\sqrt{\psi}$ is symmetric and subadditive) that $\psi(h_n)$ tends to infinity. Since this holds for every $(h_n)$, we deduce that $\psi$ is proper on $H$. Actually, the argument also works for any sequence $(f_nh_n)$ with $h_n$ leaving compact subsets of $H$; for a sequence $f_nh_n$ with $(h_n)$ bounded and $(f_n)$ leaving compact subsets of $B^{(H/L)}$, we directly use properness on $B^{(H/L)}$ to infer that $\psi(f_nh_n)$ tends to infinity.
\end{proof}

\begin{rem}
The converse of Theorem \ref{thaag} in the PW case can also naturally be asked; however I refrain to any conjecture since the analogues of the results of \cite{CT} are not available in this case.
\end{rem}

\section{Polycompact and bounded radicals}\label{poly}

Recall that in a locally compact group, 
\begin{itemize}
\item $\WW(G)$ denotes the {\bf polycompact radical}, namely the subgroup generated by all compact normal subgroups, which is also the union of all compact normal subgroups. 
\item $\FC(G)$ denotes the {\bf bounded radical}, namely the union of all relatively compact conjugacy classes ($\FC$ stands for bounded). It is also sometimes called ``topological FC-center".
\end{itemize}

We have $\WW(G)\subset\FC(G)\supset Z(G)$, where the latter is the center.
Beware that $\WW(G)$ and $\FC(G)$ can fail to be closed, see \cite{WY}, and also precisely in the case of semirestricted wreath products, see Examples \ref{wy_var} and \ref{wy_var2}.

\begin{prop}\label{wfc_wr}
Consider a semirestricted locally compact wreath product $G=B\wr^A_X H$ with $A\neq B$ and $X\neq\emptyset$. Also
\begin{itemize}
\item Write $X=X_\infty\sqcup X_{\mathrm{f}}$, separating the union of all infinite and all finite orbits.
\item Let $N$ be the kernel of the $H$-action on $X$. Define $N'\supset N$ to be the inverse image in $H$ of $\WW((H/N)_\delta)$, where $(H/N)_\delta$ denotes $H/N$ with the discrete topology (so that $\WW((H/N)_\delta)$ is the union of all finite normal subgroups of $H/N$). Define $N''\supset N$ to be the set of elements of $h$ acting on $X$ as a finitely supported permutation. 

Define $\Xi\subset H$ to be $N$ if $B$ is non-compact and $\Xi=N'\cap N''$ if $B$ is compact.

\item Define $\mathrm{Core}_B(A)$ to be the largest normal subgroup of $B$ included in $A$.
\end{itemize}
Then 
\[\WW(G)=(\mathrm{Core}_B(A)^{X_\infty}\times \WW(B)^{X_{\mathrm{fin}},\WW(B)\cap \mathrm{Core}_B(A)})\rtimes (\WW(H)\cap \Xi);\]
\[\FC(G)=(\mathrm{Core}_B(A)^{X_\infty}\times \FC(B)^{X_{\mathrm{fin}},\FC(B)\cap \mathrm{Core}_B(A)})\rtimes (\FC(H)\cap \Xi).\]

In particular, if $X$ has no finite $H$-orbit, then 
\[\WW(G)=\mathrm{Core}_B(A)^{X}\rtimes (\WW(H)\cap N);\quad\FC(G)=\mathrm{Core}_B(A)^{X}\rtimes (\FC(H)\cap N).\]
\end{prop}
\begin{proof}
Modding out by the compact normal subgroup $\mathrm{Core}_B(A)$, we can reduce to assume $\mathrm{Core}_B(A)=1$ and prove, in this case, that 
\[\WW(G)=\WW(B)^{(X_{\mathrm{fin}})}\rtimes (\WW(H)\cap \Xi);\qquad \FC(G)=\FC(B)^{(X_{\mathrm{fin}})}\rtimes (\FC(H)\cap \Xi).\]

Let $M$ be the group on the right side. We first check the inclusions $\supset$ in both cases. This follows from the following inclusions:

\begin{itemize}

\item $\WW(B)^{(X_{\mathrm{fin}})}\subset \WW(G)$. For every compact normal subgroup $K$ of $B$ and finite $H$-invariant subset $F\subset X$, $\WW(B)^F$ is a compact normal subgroup and hence is included in $\WW(G)$; it follows that the union of all these subgroups ($F$ ranging over finite $H$-invariant subsets of $X$ and $K$ among all compact normal subgroups of $B$), which is precisely $\WW(B)^{(X_{\mathrm{fin}})}$, is included in $\WW(G)$.

\item $\FC(B)^{(X_{\mathrm{fin}})}\subset \FC(G)$. Let $f$ belong to this subgroup. Then the support of $f$ being finite and included in $X_{\mathrm{fin}}$, the union of $H$-orbits is a finite $H$-invariant subset $F$. If $I$ is the finite image of $f$ and $J$ is the union of closures of conjugacy classes of $I$, then $J$ is compact and the conjugacy class of $f$ is included in $J^F$, which is compact.

\item $\WW(H)\cap N\subset \WW(G)$: indeed this is the union of $K\cap N$ where $K$ ranges over compact normal subgroups of $H$. Since $K\cap N$ is also normal in $G$, it is therefore included in $\WW(G)$. Taking the union yields the inclusion.

\item $\FC(H)\cap N\subset \FC(G)$: if $f\in N$, then it centralizes $B^{X,A}$ and hence its $H$-conjugacy class equals its $G$-conjugacy class. Hence if $f\in\FC(H)\cap N$ then $f\in\FC(G)$.

\item if $B$ is compact, $\WW(H)\cap N'\cap N''\subset \WW(G)$. Let $h$ be an element in $N'\cap N''\cap \WW(H)$. Let $M$ be the closure of the normal subgroup of $H$ generated by $h$.
 Since $h\in N''$, $M$ acts on $X$ by finitely supported permutations. Since $h\in N'$, the image of $M$ in the group of permutations of $X$ is finite. Combining, the union $Y$ of all supports of elements of $M$ is a finite subset of $X$. Since $M$ is normal in $H$, $Y$ is $H$-invariant. Since $h\in \WW(H)$, $M$ is compact. Hence $B^Y\rtimes M$ is a compact normal subgroup of $G$. Hence $h\in \WW(G)$.

\item if $B$ is compact, $\FC(H)\cap N'\cap N''\subset \FC(G)$. We argue as in the previous case. The difference is that $M$ is not necessarily compact; however, $M/(M\cap N)$ is still finite. In $B^Y\rtimes M$, modulo the compact normal subgroup $B^Y$, the $G$-conjugacy class of $f$ has compact closure; hence this holds in $B^Y\rtimes M$ and hence in $G$. 
\end{itemize}

This shows the inclusions $\supset$.
We now have to prove the reverse inclusions. Denote by $\pi$ the projection $G\to H$.

\begin{itemize}
\item $\WW(G)\subset\pi^{-1}(\WW(H))$ and $\FC(G)\subset\pi^{-1}(\FC(H))$ is clear.

\item $\WW(G)\subset\FC(G)\subset\pi^{-1}(N'\cap N'')$. Indeed, let $C$ be a conjugacy class included in $\FC(G)$. For $g\in B^{X,A}$ and $fh\in C$, we have $g(fh)g^{-1}=gf(hgh^{-1})^{-1}h$. Define $\delta_x(b)$ be the function mapping $x$ to $b\in B$ and other elements of $X$ to 1, and apply this to $g=\delta_x(b)$. Then $\delta_x(b)fh\delta_x(b)^{-1}=\delta_x(bf(x))\delta_{hx}(b^{-1})$. Fix $b\notin A$. Then, when $fh$ ranges over $C$ and $x$ ranges over elements such that $hx\neq x$, then the $x$-projection of $\delta_x(b)fh\delta_x(b)^{-1}$ is $b^{-1}$, and thus $x$ has to range over finitely many elements only. In other words, the set $Z$ of $x$ such that $h(x)\neq x$ for some $fh\in C$ is finite. This means that for every $fh\in C$, $h$ acts as the identity on the complement of $Z$. This shows that $C\subset\pi^{-1}(N'\cap N'')$.

\item if $B$ is non-compact, $\WW(G)\subset\FC(G)\subset\pi^{-1}(N)$. We prove the complementary inclusion: suppose that $fh\in G$ with $h\notin N$. Fix $x$ such that $hx\neq x$. Then
\[p:=[fh,\delta_x(b)]=fh\delta_x(b)h^{-1}f^{-1}\delta_x(b^{-1})=f\delta_{hx}(b)f^{-1}\delta_x(b^{-1})\]\[=\delta_{hx}(f(hx)bf(hx)^{-1})\delta_x(b^{-1}).\]
So $p(x)=b^{-1}$. Since $B$ is non-compact, this shows that these commutators, when $b$ ranges over $B$, do not have a compact closure and thus $fh\notin \FC(G)$.
\end{itemize}

Now we know that the projection to $H$ is exactly $\WW(H)\cap\Xi$, we only have to show the reverse inclusions for $f\in \WW(G)\cap B^{X,A}$, namely $\WW(G)\cap B^{X,A}\subset\WW(B)^{(X_{\mathrm{fin}})}$ and $\FC(G)\cap B^{X,A}\subset\FC(B)^{(X_{\mathrm{fin}})}$. This is equivalent to the three inclusions below.

\begin{itemize}
\item $\WW(G)\cap B^{X,A}\subset\WW(B)^{X,A}$ and $\FC(G)\cap B^{X,A}\subset\FC(B)^{X,A}$: both inclusions are immediate.
\item $\WW(G)\cap B^{X,A}\subset \FC(G)\cap B^{X,A}\subset B^{(X_{\mathrm{fin}})}$. Let $f$ belong to $\FC(G)\cap B^{X,A}$. Suppose by contradiction that $f$ is not supported by $X_{\mathrm{fin}}$. Then there exists $x\in X_\infty$ such that $f(x)\neq 1$. Conjugating $f$ by some element of the form $\delta_x(b)$, we can suppose that $f(x)\notin A$. Hence the $H$-conjugates of $f$ do not have compact closure, a contradiction. This shows that $f$ is supported by $X_{\mathrm{f}}$. 
\end{itemize}

To see that the last statement follows, we need to check that $N'\cap N''=N$ in case there is no finite orbit. Indeed, $(N'\cap N)$ is the union of all subgroups of the form $M$ where $M$ includes $N$ with finite index, $M/N$ is normal in $H/N$ and acts on $X$ with finite support. Given such an $M$, the union of all its supports is finite and $H$-invariant, hence empty, so $M=N$ and finally $N\cap N'=N$.
\end{proof}

\begin{cor}\label{wg0}
$\WW(G)=1$ if and only if $\WW(H)\cap \Xi=1$, ($X_{\mathrm{fin}}\neq\emptyset$ $\Rightarrow$ $\WW(B)=1$), and $\mathrm{Core}_A(B)=1$.

In case $X_{\mathrm{fin}}=\emptyset(\neq X)$, this can be restated as: $\WW(G)=1$ if and only if $\WW(H)\cap N=1$ and $\mathrm{Core}_A(B)=1$.
\end{cor}

\begin{cor}
Assume that $X_{\mathrm{fin}}$ is infinite. Then $\WW(G)$ is closed if and only if $\WW(B/\mathrm{Core}_B(A))=1$ and $\WW(H)\cap N$ is closed.
\end{cor}
\begin{proof}
Since both properties are unchanged when modding out by the compact normal subgroup $\mathrm{Core}_B(A)$, we can suppose that $\mathrm{Core}_B(A)=1$, and the result then immediately follows.
\end{proof}

\begin{rem}Note that ($\WW(G)$ is closed) $\Leftrightarrow$ ($\FC(G)$ is closed) holds for arbitrary locally compact groups. Indeed $\Leftarrow$ is stated in \cite[Prop.\ 2.4(ii)]{Cind}, but the converse $\Rightarrow$ follows from \cite[Prop.\ 2.4(iv)]{Cind}.
\end{rem}

\begin{rem}The assumption only excludes two trivial cases, that is $A=B$ or $X=\emptyset$. Actually, $\WW(G)=A^X\rtimes \WW(H)$ if $A=B$ and $\WW(G)=\WW(H)$ if $X=\emptyset$, and similarly for $\FC(G)$.
\end{rem}

\begin{ex}\label{wy_var}
Let $(F_i)$ be a family of finite groups, each $F_i$ acting faithfully on some finite set $X_i$. Then the compact group $H=\prod F_i$ acts on $X=\bigsqcup X_i$. Then in this case, $N'=N''=\bigoplus F_i$. So if $B$ is any nontrivial finite group, we have $\WW(B\wr_X H)=B^{(X)}\rtimes\bigoplus F_i=\bigoplus B\wr F_i$. In this case $\WW(G)$ is not closed. This example is very similar to the classical original example of Wu-Yu \cite{WY}. 
\end{ex}

\begin{ex}\label{wy_var2}
Let $B$ be a finite group and $A$ a proper subgroup with trivial core (e.g., $B$ is non-abelian of order 6 and $A$ has order 2). Let $X$ be a faithful $\Z$-set with only finite orbits. Then for $G=B\wr^A_X\Z$, both $\WW(G)$ and $\FC(G)$ are equal to $B^{(X)}$ (which is not closed).
\end{ex}

\section{Locally compact groups of intermediate growth}\label{intgro}

\subsection{The construction}

It is natural to wonder whether there exist CGLC (compactly generated locally compact) groups of intermediate growth that are not too close to discrete groups.

This naturally led to the following question: {\it is every CGLC group of subexponential growth compact-by-Lie}? The point is that the answer is positive for groups of polynomial growth \cite{Los}.

A construction of Bartholdi-Erschler \cite{BE} along with the semirestricted wreath product construction leads to a negative answer:

\begin{thm}\label{bethm}
There exists a totally disconnected CGLC group of intermediate growth that is not compact-by-discrete.
\end{thm} 
\begin{proof}
Bartholdi and Erschler \cite{BE} have shown that the first Grigorchuk group $\Gamma$, which has intermediate growth, has a subgroup $\Lambda$ of infinite index such that for every finite group $F$, the wreath product $\Delta=F\wr_{\Gamma/\Lambda}\Gamma$ has intermediate growth as well. 

Then assume that $n=|F|\ge 3$. Embed, by Proposition \ref{copcip} the latter as a cocompact lattice into the semirestricted wreath product $G=\mathfrak{S}_n\wr^{\mathfrak{S}_{n-1}}_{\Gamma/\Lambda}\Gamma$. Then the latter has $\WW(G)=1$ by the second case of Corollary \ref{wg0} (using either that $\WW(\Gamma)=1$ or that $\Gamma$ acts faithfully on $\Gamma/\Lambda$). Since $\Delta$ has intermediate growth, so does $G$.
\end{proof}
 
We leave the following three questions open.
 
\begin{que}\label{q_intg}
Does there exist a totally disconnected CGLC group of subexponential growth $G$ that satisfies one of the following
\begin{enumerate}
\item\label{nodi} $G$ is non-compact and has no infinite discrete quotient (as a topological group)?
\item\label{nocom} $G$ is not commable to any discrete group?
\item\label{rewe} (Caprace \cite[Question 3.9]{Capr1}) $G$ is topologically simple and non-discrete? (Wesolek \cite[Problem 20.9.6]{Capr}) $G$ is not Wesolek-elementary? (Wesolek-elementary \cite{We} means that it lies in the smallest isomorphism-closed class of locally compact groups containing the trivial groups and stable under taking directed unions and extensions with discrete or profinite quotients.)
\end{enumerate}
\end{que}

Recall that commable is the ``equivalence relation" between locally compact generated by the relations: $G_2$ is isomorphic to a the quotient of a normal compact subgroup of $G_1$, and: $G_2$ is isomorphic to a closed cocompact subgroup of $G_1$.

\begin{rem}
1) If $G$ has polynomial growth, it is compact-by-discrete \cite{Los} and hence cannot fulfill any of the requirements (this also follows as a consequence of Trofimov's results on graphs \cite{Tro3}).

2) If $G$ is a non-compact totally disconnected CGLC group, by \cite[Theorem A]{CM}, either it has an infinite discrete quotient, or it has a non-compact non-discrete topologically simple, compactly generated subquotient $S$. If $G$ has intermediate growth, then so does $S$; in particular a positive answer to (\ref{nodi}) is equivalent to the existence of a non-discrete CGLC, topologically simple group of intermediate growth; thus a positive answer to (\ref{nodi}) would also answer (\ref{rewe}).
\end{rem}

See \S\ref{otherw} for an alternative construction for Theorem \ref{bethm}, yielding possible candidates for Question \ref{q_intg}(\ref{nocom}).

\begin{rem}
It is well-known that a discrete-by-compact CGLC group $G$ is automatically compact-by-discrete. Indeed, let $D$ be a cocompact normal discrete subgroup. Being finitely generated, its centralizer $C$ is open. Since clearly $G^\circ$ is compact, $C$ admits an open compact subgroup $K$, and hence $KD$ is a cocompact open subgroup, thus has finite index. So the intersection of all conjugates of $K$ is also open, and thus $G$ is compact-by-discrete.

It is also well-known that this implication does not hold for arbitrary $\sigma$-compact locally compact groups: the group of Example \ref{wy_var} is typical counterexample.
\end{rem}

\subsection{Trofimov's conjecture}\label{trocon}

We say that a connected graph $X$ {\bf essentially includes a tree} if there exists an injective, Lipschitz map from the regular trivalent tree into $X$ (Trofimov calls this ``hyperbolic" but this differs from usual terminology).

In the following, a graph is identified with its vertex set (so the only graph structure that matters is the given adjacency relation between vertices); in particular a graph homomorphism is understood to be a map between vertex sets mapping adjacent vertices to adjacent or equal vertices; it also means a 1-Lipschitz map. We call it a graph-quotient homomorphism if every edge from $Y$ is image of an edge in $X$.

Let $G$ be a group acting on a connected graph $X$. We say that the action is {\bf block-discrete} if there exists a graph $Y$ and a continuous action of $G$ on $X$, a surjective $G$-equivariant graph-quotient homomorphism $X\to Y$ with finite fibers, such that, denoting by $G^Y$ the image of $G$ in $\Aut(Y)$, the vertex stabilizers of the $G^Y$-action on $Y$ are finite.

Trofimov's conjecture can now be stated: 

\begin{conj}[Trofimov \cite{Tro1,Tro2}]Let $G$ be a group acting vertex-transitively on a connected graph $X$ of finite valency. Then either the action is block-discrete, or $X$ essentially includes a tree.
\end{conj}

In particular, it predicts that in the case $X$ has subexponential growth, the action is block-discrete. On the other hand, we have the following easy observation:

\begin{fact}
Let $X$ be a connected graph of finite valency. Let $G$ be a locally compact group acting properly vertex-transitively on $X$. Then the $G$-action is block-discrete if and only if $G$ has a compact open normal subgroup.
\end{fact}
\begin{proof}
We only prove the implication we need, leaving the converse to the reader. Suppose that the action is block-discrete, and let $Y$ be as in its definition. Since the vertex stabilizers in $Y$ include vertex stabilizers in $X$, the $G$-action on $Y$ is continuous. Since the $G$-action on $X$ is proper and fibers are finite (and since $X\to Y$ is a graph-quotient homomorphism), the $G$-action on $Y$ is proper.
Let $W$ be the kernel of the $G$-action on $Y$ and let $H$ be a vertex-stabilizer. Since $H/W$ is finite, $H$ is open and $W$ is a closed subgroup of $H$, the normal subgroup $W$ is open as well. By properness, $W$ is compact, proving the implication.
\end{proof}

We can conclude, as a Corollary of Theorem \ref{bethm}:

\begin{cor}
Trofimov's above conjecture is not true.
\end{cor} 
\begin{proof}
Let $G$ be a totally disconnected CGLC group of intermediate growth and no compact open subgroup, as asserted in the theorem. Let $X$ be a Cayley-Abels graph for $G$ (see \cite[\S 2.E]{CH}): this is a connected graph of finite valency on which $G$ acts continuously, properly and vertex-transitively. Since $G$ is quasi-isometric to $X$, the graph $X$ has subexponential growth, and in particular does not essentially include a tree in the above sense. By the fact above, the action is not block-discrete. So it does not satisfy the conjecture.
\end{proof}

\begin{rem}
Part of the discussion in \cite{Tro2} is about when the above conjecture is specified when the action of the vertex stabilizer $G_v$ on the 1-sphere around one vertex $G_v$ is specified to be, modulo its kernel, a given finite permutation group. We have not tried to describe this permutation group in this construction. However, we can at least say something: we can arrange the counterexample so that the group in this finite permutation group is a 2-group. Indeed, it is enough to construct $G$ as a 2-group. Using the notation from the proof of Theorem \ref{bethm}, we can make a slight change in the construction and assume that $G=F\wr_{\Gamma/\Lambda}^L\Gamma$, where $F$ is a nontrivial finite 2-group and $L$ is a nontrivial subgroup of $F$ with trivial core (e.g., $F$ the dihedral group of order 8 and $L$ a non-central cyclic subgroup of order 2). Then $\mathsf{W}(G)=1$ by Corollary \ref{wg0}, and $G$ has intermediate growth for the same reason as in the proof of Theorem \ref{bethm}. (Alternatively, we can use the construction of Proposition \ref{autreinter}, which yields a 2-group.)

In \cite[Remark 2.3]{Tro2}, Trofimov says that he would be ``rather surprised if [his] conjecture were proved in general", but is optimistic about the case when the local permutation group is primitive. We do not even know if our example can be arranged to even yield transitive local permutation groups (beware that this depends on the choice of Cayley-Abels graph, and forces studying the structure of the Grigorchuk group action, rather than using it as a black box).
\end{rem}

\section{Presentability}\label{presen}

The material of this section is essentially borrowed from an expunged part, appearing in an earlier version of \cite{BCGS} (Arxiv v2), exclusively for discrete groups, in keeping with \cite{BCGS}.

We first introduce the following definition, which is \cite[Def.\ 5.9]{BCGS} in the discrete case.

\begin{defn}\label{dlr}
A CGLC group $H$ is {\bf largely related} if for every epimorphism $G\epi H$ of a compactly presented locally compact group $G$ onto $H$ with discrete kernel, the kernel admits a non-abelian free quotient.
\end{defn}

\begin{defn}
A family $(N_i)_{i\in I}$ of closed normal subgroups of a locally compact group $G$ is {\bf independent} if $N_i$ is not included in $\langle N_j:j\in I\smallsetminus\{i\}\rangle$ for any $i\in I$, or equivalently if the map from $2^I$ to the space $\mathcal{N}(G)$ of closed normal subgroups of $G$ mapping $J$ to $\langle N_j:j\in J\rangle$ is injective. 

(\cite[Def.\ 1.2]{BCGS}) A CGLC group $G$ is {\bf INIP} (infinitely independently presented) if for some/every compactly presented locally compact group $G_0$ with a quotient map $G_0\epi G$ with discrete kernel, the kernel is generated by an infinite independent family of $G_0$-normal subgroups.
\end{defn}

\begin{defn}\label{d_sim}Let $H$ be a group and $L\subset H$ a subgroup. Consider the equivalence relation on $H$: $g_1\sim g_2$ if $g_1$ belongs to the same $L$-double coset as $g_2$ or $g_2^{-1}$. Let $Q$ be the quotient of $H$ by this equivalence relation and $Q^*=Q\smallsetminus\{L\}$ (observe that $L$ is a single equivalence class).
\end{defn}

Now assume that $H$ is a locally compact group, $L$ an open subgroup, $B$ another locally compact group and $A$ a compact open subgroup of $B$. We need to define a locally compact group $G_0$ with a continuous quotient homomorphism with discrete kernel $G_0\to B\wr^A_{H/L}H$. 
In case $A=1$ (so $B$ is discrete and this is a usual wreath product), the definition of $G_0$ is given by the amalgamated product $(B\times L)\ast_L H$. 
As an amalgam of two locally compact groups over a common open subgroup, this is naturally a locally compact group. As an abstract group, this is the quotient of the free product $B\ast H$ by the normal subgroup generated by commutators $[B,L]$; however $B\ast H$ is not a locally compact group in a natural way unless $H$ is discrete.
In general, $G_0$ is defined as the locally compact group
\[G_0=(\Pi_{B,A,H/L})\ast_{A^{H/L}}(A\,\bar{\wr}\,H),\] where \[\Pi_{B,A,H/L}=\{f\in B^{H/L}:f(H/L\smallsetminus\{L\})\subset A\}\qquad \text{(so } \Pi_{B,A,H/L}\simeq B\times A^{H/L\smallsetminus\{L\}}).\] 

Note that $G_0$ is compactly generated as soon as $H,B$ are. From the universal property, there is a unique homomorphism $G_0\to B\wr^A_{H/L}H$ mapping both factors identically. It is continuous and surjective, and has discrete kernel $N$. (That it is continuous with discrete kernel follows from the fact that it restricts to the standard embedding of the open subgroup $A^{H/L}\rtimes L$.) 

Observe that in $G_0$, the normal subgroup $N_g$ generated by $[gBg^{-1},B]$, for $g\in H\smallsetminus L$, only depends on the class of $g$ in $Q^*$. It is easy to see that the $N_g$ generate the kernel $N$.

\begin{prop}\label{p_wr}
If $B\neq A$, then the family of normal subgroups $(N_g)$ of $G_0$ is independent, when $g$ ranges over $Q^*$. In particular, if the double coset space $L\backslash H/L$ is infinite and both $B$ and $H$ are compactly generated, then the semirestricted wreath product $G=B\wr^A_{H/L}H$ is INIP, and moreover it is largely related.
\end{prop}

This extends a result of \cite{Cgd}, where it was shown under the same assumptions (and in the discrete case) that the wreath product $B\wr_{G/H}G$ is infinitely presented.

\begin{proof}
In $G_0$, denote by $H$ and $B$ the obvious copies of these groups. For $h\in H$ and $x=hL\in H/L$, define $B_x=hBh^{-1}$, and $A_x=hAh^{-1}\subset B_x$. Note that in $G_0$, the subgroup $U$ generated by all $B_x$ is the quotient of their free product by the relations $[A_x,B_y]$ for $x\neq y\in H/L$.

Observe that $Q^*$ is obtained from $L\backslash H/L$ by removing one point and modding out an action of the cyclic group of order two (by inversion). So if $L\backslash H/L$ is infinite, so is $Q^*$. So once we will have proved that $(N_g)_{g\in Q^*}$ is independent, it will follow that the quotient $G$ is INIP.

Lift $Q^*$ to a subset of $H$. If $I$ is a subset of $Q^*$, let $N_I$ be the normal subgroup of $H$ generated by $\bigcup_{g\in I}N_g$. Then $N_I$ is generated, as a normal subgroup of $H$, by $\bigcup_{g\in I}[B_{gH},B]$. That $(N_g)$ is independent means that for every $s\in Q^*$, the group $N_{Q^*\smallsetminus\{s\}}$ is not equal to $N_{Q^*}$. We will actually show that for every $s\in Q^*$, the group $N_{Q^*}/N_{Q^*\smallsetminus\{s\}}$ is non-trivial.

By a straightforward verification, $N_{Q^*\smallsetminus\{s\}}$ is generated, as a normal subgroup of $U$, by \[\bigcup_{g\in Q^*\smallsetminus\{s\},\gamma\in H}\gamma[B_{g\xi},B_\xi]\gamma^{-1}\]
 (where $\xi$ denotes the base-point in $H/L$), or equivalently by
\begin{equation}\label{gga}\bigcup_{g\in Q^*\smallsetminus\{s\}}\bigcup_{\gamma\in H}[B_{\gamma g\xi},B_{\gamma\xi}].\end{equation}
Let $J_s$ be the normal subgroup of $U$ generated by $$\bigcup_{x\in H/L\setminus\{\xi,s\xi\}}B_x,$$
so $U/J_s$ is naturally identified with
\[\Xi=(B_\xi\ast B_{s\xi})/\langle [A_\xi,B_{s\xi}],[B_\xi,A_{s\xi}]\rangle\simeq (B_\xi\times A_{s\xi})\ast_{A_\xi\times A_{s\xi}}(A_\xi\times B_{s\xi}).\] 

The kernel $K$ of $U/J_s\to B_\xi\times B_{s\xi}$ is discrete and free: indeed it acts trivially on the Bass-Serre tree for this amalgam decomposition. If $|B/A|\ge 3$, it is non-abelian.

Observe that $N_{Q^*\smallsetminus\{s\}}$ is included in $J_s$: indeed it is generated by commutators $[B_x,B_y]$ with $\{x,y\}\neq\{\xi,s\xi\}$, so each of these commutators is contained in $J_s$. Thus there is a natural epimorphism $U/N_{Q^*\smallsetminus\{s\}}\epi B_\xi\ast B_{s\xi}$. It restricts to an epimorphism
$$N_{Q^*}/N_{Q^*\smallsetminus\{s\}}\epi K.$$

If $|B/A|=2$, we can fix the argument as follows: first in this case $A$ is normal and hence we can mod it out and hence suppose that $B$ is discrete. Then we show that if $s,t$ are distinct in $Q^*$, the group $N_{Q^*}/N_{Q^*\smallsetminus\{s,t\}}$ surjects onto a non-abelian free group, namely the kernel of the projection $B_\xi\ast (B_{s\xi}\times B_{t\xi})\epi B_\xi\times B_{s\xi}\times B_{t\xi}$.

This shows that, whenever $B\neq A$, for every subset $I\subset Q^*$ whose complement contains at least two elements, $N_{Q^*}/N_I$ has a non-abelian free quotient. 

Thus if $Q^*$ is infinite, and if $P$ is compactly presented group with an epimorphism $\pi$ onto $B\wr^A_{G/H}G$, this epimorphism factors through the projection $G_0/N_I$ for some finite $I\subset Q^{*}$. So, the kernel of $\pi$ admits $N_{Q^*}/N_I$ as a quotient and therefore possesses a non-abelian free group as a quotient. This shows that $B\wr_{G/H}G$ is largely related.
\end{proof}
 
We now turn to another similar example based on Coxeter groups. 

\begin{defn} 
Consider a {\bf Coxeter matrix} on $V$  i.e.\ a symmetric matrix $\mu:V\times V\to \{1,2,3,\dots,\infty\}$ with with diagonal entries equal to 1 and non-diagonal entries in $\{2,3,\dots,\infty\}$. It defines the Coxeter group with Coxeter presentation
$$W(V,\mu)=\left\langle (w_v)_{v\in V}\mid \left((w_sw_t)^{\mu(s,t)}\right)_{(s,t)\in V^2}\right\rangle.$$ 

Let a group $H$ act on $V$.
Now assume that $\mu$ is $H$-invariant, in the sense that $\mu(gs,gt)=\mu(s,t)$ for all $g\in H$ and $(s,t)\in V^2$. This induces a natural action of $H$ by automorphisms on $W(V,\mu)$, so that $g\cdot w_s=w_{gs}$. The corresponding semidirect product
$$W(V,\mu)\rtimes H$$
is called a {\bf wreathed Coxeter group}. 
\end{defn}

When $H$ is locally compact and acts continuously on $V$ (which is discrete), that is, with open stabilizers, then the wreathed Coxeter group $W(V,\mu)\rtimes H$ is a topological group ($W(V,\mu)$ being discrete and normal). 

This group was already considered, from a different perspective, in \cite{CSV2}, when $H$ is discrete. 

If $H$ acts with finitely many orbits on $V$ and is finitely generated discrete (resp.\ compactly generated), then the wreathed Coxeter group $W(V,\mu)\rtimes H$ is discrete finitely generated (resp.\ compactly generated).

\begin{thm}
Assume $V\neq\emptyset$. The wreathed Coxeter group $G=W(V,\mu)\rtimes H$ is compactly presented if and only if $V$ has finitely many $H$-orbits with compactly generated stabilizers, $H$ is compactly presented, and the set of pairs $\{(v,w)\in V^2:\mu(v,w)<\infty\}$ consists of finitely many $H$-orbits.

In particular, if $V=H/L$ and $\mu$ has no $\infty$ entry, then this holds if and only if $L\backslash H/L$ is finite.

If the set of pairs $\{(v,w)\in V^2:\mu(v,w)<\infty\}$ consists of infinitely many $H$-orbits, then $G$ is INIP.
\end{thm}
\begin{proof}
We only sketch the proof. Fix $v_1,\dots,v_k$ representative of the $H$-orbits in $V$, with stabilizers $L_1,\dots,L_k$. Let for $i=1,\dots,k$, let $C_i=\langle t_i\rangle$ be a copy of the cyclic group of order 2. Consider the amalgam $G_0$ of all $C_i\times L_i$ and $G$ over their intersections: it can be constructed iteratively:
\[G_0=((\cdots(G\ast_{L_1}(C_1\times L_1))\ast_{L_2}(C_2\times L_2)\cdots )\ast_{L_k}(C_k\times L_k))\]
Then $G_0$ is compactly presented as soon as $H$ is compactly presented and all $L_i$ are compactly generated. Then the wreathed Coxeter group is the quotient by the relations
is the quotient by $r_{g,h}=(gt_ig^{-1}ht_jh^{-1})^{\mu(gv_i,hv_j)}$ whenever $\mu(gv_i,hv_j)<\infty$. Actually two such relations $r_{g,h}$ and $r_{g',h'}$ are equivalent as soon as $(gv_i,hv_j)=(g'v_i,h'v_j)$. Hence if there are finitely many such orbits of pairs, then $G$ is compactly presented. 

The converse follows the same lines as the case of wreath products, relying on a well-known result of Tits, Theorem \ref{titsth}.

Because of the similarity with the proof of Proposition \ref{p_wr}, we will prove the result in a particular case that is enough to encompass all the differences, namely the case when $H=\Z=\langle t\rangle=X$ (simply transitive action). The reader is invited to prove the general case as an exercise.

So we have to prove that in the free product $\Gamma=\langle t,w|w^2=1\rangle$, the family of relators $r_n=(wt^nwt^{-n})^{\mu(0,n)}$, for $n\in\Z^1=\{n\ge 1:\;\mu(0,n)<\infty\}$, is independent.

If $p\in\Z^1$, let $\Gamma_{[p]}$ be the group obtained by modding out $\Gamma$ by all relators $r_n$ for $n\neq p$, and let $\mu'$ be the matrix obtained from $\mu$ by replacing all entries $\mu(n,n+p)$ by $\infty$.
We see that in $\Gamma_{[p]}=W(V,\mu')\rtimes\Z$. By Tits' theorem, $wt^pwt^{-p}$ has infinite order in $\Gamma_{[p]}$, so $r_p\neq 1$ in $\Gamma_{[p]}$. This proves independency of the family of relators. 
\end{proof}

\begin{thm}[Tits]\label{titsth}
Given a Coxeter group generated by involutions $(\sigma_s)_{s\in S}$, subject to relators $(\sigma_s\sigma_t)^{\mu(s,t)}$ for all $s,t$ (where $\mu(s,t)_{(s,t)\in S\times S}$ is a Coxeter matrix), the element $\sigma_s\sigma_t$ has order exactly $\mu(s,t)$, and every subgroup generated by a subset $(\sigma_s)_{s\in T}$ is a Coxeter group over this system of generators.
\end{thm}

This follows from \cite[V.\S 4.3 Prop.~4]{bki} and \cite[IV.\S 1.8 Th.~2]{bki}.

\begin{ex}
The group $\Gamma$ of permutations of $\Z$ generated by the transposition $0\leftrightarrow 1$ and the shift $n\mapsto n+1$, which is isomorphic to wreathed Coxeter group $\textnormal{Sym}_0(\Z)\rtimes\Z$, is INIP; here $\textnormal{Sym}_0(\Z)$ denotes the group of finitely supported permutations of $\Z$. That the group $\Gamma$ is infinitely presented is implicit in B.H.\ Neumann \cite{Neu}, who expressed it as quotient of a finitely generated group by a properly increasing union of finite normal subgroups. Half a century later, it was mentioned by St\"epin \cite{stepin1983} as an example of a finitely generated group that is {\bf locally embeddable in finite groups} in the sense of Maltsev (this also means: approximable by finite groups for the topology of the space of marked groups) but is not residually finite, a combination that cannot be achieved by finitely presented groups.
\end{ex}

\section{Variants using commensurating actions}\label{otherw}

Let $X$ be a set and  $\mathcal{M}=(M_i)_{i\in I}$ be partition of $X$ (i.e., pairwise disjoint and covering $X$). 
Let $B$ be a locally compact group and $\mathcal{A}=(A_i)_{i\in I}$ a family of compact open subgroups. Consider the subgroup of $B^X$ generated by its subgroups $\prod_{i\in I}{A_i}^{M_i}$ and $B^{(X)}$. Denote it by $B^{X,\mathcal{M},\mathcal{A}}$. 

For instance, if $I$ is a singleton $\{1\}$ and $M_1=X$, then this is precisely $B^{X,A}$. The main motivating case is when $I$ has two elements.

The $B^{X,\mathcal{M},\mathcal{A}}$ is endowed with the group topology making $\prod_{i\in I}{A_i}^{M_i}$ a compact open subgroup. It is standard (e.g., follows from Lemma \ref{extop}) that this is well-defined.

Let now $H$ be a locally compact group and assume that $X$ is a continuous discrete $H$-set. 

Assume that the family $\mathcal{M}$ is {\bf uniformly commensurated} by $H$, in the sense that for every $h\in H$, we have ($\bigtriangleup$ denoting symmetric difference):
\[\sum_i\#(M_i\bigtriangleup hM_i)<\infty.\]

Note that if $H$ is finitely generated, then this forces all but finitely many of the $M_i$ to be $H$-invariant; this can be extended to the case when $H$ is compactly generated (see Proposition \ref{cg_fin}), but not in general. For instance, for $H=\Q_2$ and $X=\Q_2/\Z_2$, one can consider the family (modulo $\Z_2$) indexed by $\N$: $M_0=\Z_2/\Z_2=\{0\}$, $M_i=(2^{-i}\Z_2\smallsetminus 2^{-i+1}\Z_2)/\Z_2$ for $i\ge 1$.

Under the above assumptions, the action of $H$ on $B^X$ preserves $B^{X,\mathcal{M},\mathcal{A}}$, and is continuous. Therefore the semidirect product 
\[B\wr_X^{\mathcal{M},\mathcal{A}}H=B^{X,\mathcal{M},\mathcal{A}}\rtimes H\]
 is a locally compact group.

Let us focus on a specified case $I=\{1,2\}$, $B$ is a finite group, $A_1=\{1\}$ and $A_2=B$; we can only specify $M=M_2$ since $M_1$ is its complement. Then denote $B^{X,\mathcal{M},\mathcal{A}}=B^{(X,M]}$: it is just the direct product $B^{(X\smallsetminus M)}\times B^M$. The assumption of commensuration reduces to the requirement that $M$ is commensurated by the $H$-action: $M\bigtriangleup hM$ is finite for all $h\in H$. Denote $B\wr_{(X,M]}H=B^{(X,M]}\rtimes H$ and call it {\bf half-restricted wreath product}. This important particular case was introduced by Kepert and Willis \cite{KW}. It was used by Bhattacharjee and Macpherson \cite{BM} to exhibit a compactly generated totally disconnected locally compact group that is uniscalar but has no open compact normal subgroup.

\begin{ex}
Here is one particular case where a group naturally occurs and actually turns out to be a half-restricted wreath product. Let $\K$ be the field $\mathbf{F}_q(\!(t)\!)$ of Laurent series over the finite field $\mathbf{F}_q$. Then the affine group $\K\rtimes\K^*$ over $\K$ can naturally be identified with
\[\mathbf{F}_q\wr_{(X,M]}\K^*,\]
where $X=\K^*/\K^*_1\simeq\Z$ (quotient by the subgroups of elements of modulus 1) and $M$ is the image of the closed ball of radius 1 (corresponding in $\Z$ to the set $\N$ of non-negative integers).

This shows that this construction can produce non-unimodular groups, in contrast with Proposition \ref{unimod}.
\end{ex}

\begin{ex}\label{envw2}
Consider $X=\Z\times\{1,2\}$, $M_1=(\N\times\{1\})\sqcup((\Z\smallsetminus\N)\times\{2\}$ and $M_2$ its complement. Choose $B=\mathbf{F}_q$ (finite field), $A_1=B$, $A_2=\{0\}$; finally let $\Z$ act on $X$ by $n\cdot (m,i)=(m+n,i)$. 
 Then $B\wr_{(X,M]}H$ can be identified with the semidirect product $(\mathbf{F}_q(\!(t)\!)\times\mathbf{F}_q(\!(t^{-1})\!))\rtimes\Z$, where the positive generator of $\Z$ acts by multiplication by $t$ on both sides. This group naturally includes a cocompact lattice isomorphic to the lamplighter group $\mathbf{F}_q\wr\Z$.
\end{ex}

We now address the description of $\WW(G)$ and $\FC(G)$. For simplicity, let us reduce the study of $\WW(G)$ and $\FC(G)$ to the case with no finite orbit.

\begin{prop}\label{w_h_w}
Suppose that $H$ has a single infinite orbit on $X$. Let $N$ be the kernel of the $H$-action on $X$. Define $C=\bigcap_{i\in B,b\in B}gA_ig^{-1}=\bigcap_{i\in I}\mathrm{Core}_B(A_i)$.

Then, for $G=B\wr_X^{\mathcal{M},\mathcal{A}}H$, we have $\WW(G)=C^X\rtimes (N\cap\WW(H))$ and $\FC(G)=C^X\rtimes (N\cap\FC(H))$.
\end{prop}
\begin{proof}
Since $C$ is compact normal in $B$, $C^X$ is compact normal in $G$, and is included in all the terms considered; hence we can suppose that $C=1$. 

The subgroup $N$ of $H$ is normal, and on $N$ the $H$- and $G$-conjugacy classes coincide. It immediately follows that $N\cap\WW(H)\subset\WW(G)$ and $N\cap\FC(H)\subset\FC(G)$. 

Let $\pi$ be the projection $G\to H$. Clearly $\pi(\WW(G))\subset\WW(H)$ and $\pi(\FC(G))\subset\WW(FC)$.

Let us show that $\FC(G)\cap B^{X,\mathcal{M},\mathcal{A}}=1$. By contradiction, let $f$ be a nontrivial element. Then $c=f(x)\neq 1$ for some $x$. Then there exists $i$ such that $c\notin\mathrm{Core}_B(A_i)$. Conjugating if necessary, we can suppose $c\notin A_i$. Then by transitivity, the $H$-conjugates of $f$ do not remain in a compact subset, and this is a contradiction.  

Let us show that $\pi(\FC(G))\subset N$. Let $fh$ be an element of $\FC(G)$ with $h\in H$ and $f\in B^{X,\mathcal{M},\mathcal{A}}$. If $\pi(h)\notin N$, there exists $x\in X$ such that $h(x)\neq x$. For any $b\in B$, let $\delta_x(b)$ be defined as in the proof of Proposition \ref{wfc_wr}. Then, writing $c=f(x)$ 
\[(fh)^{-1}\delta_x(b)fh\delta_x(b)^{-1}=h^{-1}\delta_x(c^{-1}bc)h\delta_{x}(b^{-1})=\delta_{h^{-1}x}(c^{-1}bc)\delta_{x}(b^{-1});\]
this is a nontrivial element of $\FC(G)\cap B^{X,\mathcal{M},\mathcal{A}}$, which is a contradiction.

Gathering everything, $\WW(G)$ is a subgroup whose projection to $H$ is equal to $\WW(H)\cap N$, it includes $\WW(H)\cap N$, and its intersection with $B^{X,\mathcal{M},\mathcal{A}}$ is trivial. Hence $\WW(G)=\WW(H)\cap N$. Similarly $\FC(G)=\FC(H)\cap N$.
\end{proof}

Let $\Gamma$ be the first Grigorchuk group and $\Lambda$ its subgroup as in Theorem \ref{bethm}. Then the Schreier graph $\Gamma/\Lambda$ is known to be 2-ended. Pick one half $M$. Denote by $C_2$ the cyclic group on 2 generators.

\begin{prop}\label{autreinter}
The embedding $C_2\wr_{\Gamma/\Lambda}\Gamma\to C_2\wr_{(\Gamma/\Lambda,M]}\Gamma$ has a dense image. In particular, $G=C_2\wr_{(\Gamma/\Lambda,M]}\Gamma$ has intermediate growth. It has $\WW(G)=\FC(G)=1$, and in particular is not compact-by-discrete.
\end{prop}
\begin{proof}
The density is clear and the only nontrivial point is Bartholdi-Erschler's theorem that the left-hand discrete group has intermediate growth. It follows that the right-hand group has subexponential growth (bounded above by that of the discrete one). It does not have polynomial growth, since its quotient $\Gamma$ does not. It has $\FC(G)=1$, by Proposition \ref{w_h_w}.\end{proof}

This example is motivated by Question \ref{q_intg}(\ref{nocom}): unlike the examples in the proof of Theorem \ref{bethm}, which by construction have a cocompact lattice, these ones do not a priori (note that this is not a single group: several examples are provided in \cite{BE} and also the choice of $M$ rather or its complement could a priori matter).

\begin{prop}\label{cg_fin}
In the setting above (beginning of the section), if $H$ is compactly generated (or more generally, has {\bf uncountable cofinality}, in the sense that it is not the union of a properly increasing sequence of subgroups), then all but finitely many of the $M_i$ are $H$-invariant.
\end{prop}
\begin{proof}
Define $X'=X\times\{i\}$, with the component-wise action of $H$. Define $M=\bigcup_{i\in I}M_i\times\{i\}$. Then $M$ is commensurated by the $H$-action. Since $G$ is compactly generated, $M$ intersects only finitely many orbits in a non-invariant subset \cite[Prop.\ 4.B.2]{CFW}. Hence $M\cap (X\times\{i\})$ is $H$-invariant for all but finitely many $i$, and this precisely means that $M_i$ is $H$-invariant.
\end{proof}

\end{document}